\documentclass[a4,11pt]{amsart}
\usepackage{amssymb,amstext,amsmath,amscd,amsthm,amsfonts,enumerate,graphicx,latexsym}
\usepackage[usenames]{color}
\usepackage[all]{xy}
\newtheorem{theorem}{Theorem}[section]
\newtheorem*{theorem*}{Main Theorem}
\newtheorem{corollary}[theorem]{Corollary}
\newtheorem{lemma}[theorem]{Lemma}
\newtheorem{proposition}[theorem]{Proposition}

\theoremstyle{definition}
\newtheorem{definition}[theorem]{Definition}

\newtheorem{condition}[theorem]{Condition}

\newtheorem*{question*}{Question}
\newtheorem{question**}{Question}

\newtheorem*{conjecture*}{Conjecture}
\newtheorem{example}[theorem]{Example}

\newtheorem*{claim*}{Claim}
\newtheorem{mainthm}{Theorem}

\newtheorem{maincor}{Corollary}

\numberwithin{equation}{theorem}

\def\Ext{\operatorname{Ext}}

\def\Db{\mathsf{D^b}}
\def\Kb{\mathsf{K^b}}
\def\ds{\mathsf{D_{sg}}}

\def\Hom{\operatorname{Hom}}

\def\mod{\operatorname{\mathsf{mod}}}
\def\proj{\operatorname{\mathsf{proj}}}
\def\inj{\operatorname{\mathsf{inj}}}

\def\tri{\operatorname{\mathsf{tri}}}
\def\add{\operatorname{\mathsf{add}}}

\newcommand{\op}{\mathsf{op}}

\newcommand{\A}{\mathcal{A}}\newcommand{\B}{\mathcal{B}}
\newcommand{\C}{\mathcal{C}}

\renewcommand{\S}{\mathcal{S}}
\newcommand{\T}{\mathcal{T}}
\newcommand{\U}{\mathcal{U}}
\newcommand{\V}{\mathcal{V}}
\newcommand{\X}{\mathcal{X}}

\newcommand{\Ker}{\operatorname{Ker}}
\renewcommand{\Im}{\operatorname{Im}}

\newcommand{\pd}{\mathsf{pd}}
\newcommand{\id}{\mathsf{id}}
\newcommand{\gd}{\mathsf{gl.dim}}

\newcommand{\xto}{\xrightarrow}

\begin{document}
\setlength{\baselineskip}{15pt}
\title[Singular equivalences of functor categories]{Singular equivalences of functor categories via Auslander-Buchweitz approximations}
\author{Yasuaki Ogawa}
\address{Graduate School of Mathematics, Nagoya University, Furo-cho, Chikusa-ku, Nagoya, 464-8602, Japan}
\email{m11019b@math.nagoya-u.ac.jp}

\maketitle

\begin{abstract}
The aim of this paper is to construct singular equivalences between functor categories.
As a special case, we show that there exists a singular equivalence arising from a cotilting module $T$,
namely, the singularity category of $(^\perp T)/[T]$ and that of $(\mod A)/[T]$ are triangle equivalent.
In particular, the canonical module $\omega$ over a commutative Noetherian ring induces a singular equivalence between $(\mathsf{CM}R)/[\omega]$ and $(\mod R)/[\omega]$, which generalizes Matsui-Takahashi's theorem.
Our result is based on a sufficient condition for an additive category $\A$ and its subcategory $\X$ so that the canonical inclusion $\X\hookrightarrow\A$ induces a singular equivalence $\ds(\A)\simeq \ds(\X)$,
which is a functor category version of Xiao-Wu Chen's theorem.
\end{abstract}

\section{Introduction}\label{sect:1}
Let $\A$ be an additive category with weak-kernels.
Then the functor category $\mod\A$, the category of finitely presented contravariant functors from $\A$ to the category of abelian groups, is abelian.
The notion of \textit{singularity category of $\A$} is defined to be the Verdier quotient
$$\ds(\A):=\frac{\Db(\mod\A)}{\Kb(\proj\A)},$$
where we denote by $\mathsf{D}^{\textnormal{b}}(\mod\A)$ the bounded derived category, and by $\mathsf{K}^{\textnormal{b}}(\proj\A)$ the homotopy category of bounded complexes whose terms are projective.
This concept was introduced as a homological invariant of rings by Buchweitz \cite{Buc86}.
Recently it was applied by Orlov to study Landau-Ginzburg models \cite{Orl04}.
A lot of studies on singularity categories has been done in various approaches (e.g. \cite{Iya18, KV, Orl09, Ric, Zim}).

For additive categories $\A$ and $\A'$ with weak-kernels,
we say that $\A$ is \textit{singularly equivalent} to $\A'$ if there exists a triangle equivalence $\ds(\A)\simeq\ds(\A')$ \cite{ZZ}.
If $\Lambda $ is an Iwanaga-Gorenstein ring, then the singularity category of $\Lambda $ is triangle equivalent to the stable category of Cohen-Macaulay $\Lambda $-modules.
Thus the singular equivalence is a generalization of the stable equivalence for Iwanaga-Gorenstein rings.

It is a basic problem to compare homological propeties of a ring $\Lambda$ with its subalgebra $e\Lambda e$ given by an idempotent $e\in\Lambda$ (e.g. \cite{APT92, CPS, DR}).
In this context, Xiao-Wu Chen investigated a sufficient condition for a ring $\Lambda $ and its idempotent subalgebra $e\Lambda e$
so that they induce a triangle equivalence $\ds(\Lambda )\xto{\sim}\ds(e\Lambda e)$ \cite[Thm. 1.3]{Che}.
The first aim of this article is to provide its functor category version by using the following observations on Serre and Verdier quotients:
Let $\X$ be a contravariantly finite subcategory of an additive category $\A$ with weak-kernels.
Then $\X$ also admits weak-kernels, hence the canonical functor $Q: \mod\A\to\mod\X$ induces an equivalence
\begin{equation}\label{left_recollement}
\frac{\mod\A}{\mod(\A/[\X])}\xto{\sim}\mod\X,
\end{equation}
where the fraction denotes the Serre quotient (e.g. \cite[Prop. 3.9]{Buc97}).
Moreover, the equivalence (\ref{left_recollement}) induces a triangle equivalence
\begin{equation}\label{left_recollement2}
\frac{\mathsf{D}^{\textnormal{b}}(\mod\A)}{\mathsf{D}^{\textnormal{b}}_{\A/[\X]}(\mod\A)}\xto{\sim}\mathsf{D}^{\textnormal{b}}(\mod\X),
\end{equation}
where $\mathsf{D}^{\textnormal{b}}_{\A/[\X]}(\mod\A)$ is a thick subcategory consisting of objects whose cohomologies belong to $\mod(\A/[\X])$ (see \cite[Thm. 3.2]{Miy} and \cite[Thm. 2.3]{CPS}).
The equivalence (\ref{left_recollement2}) gives the following first result of this paper.

\begin{mainthm}[Lemma \ref{lem:Chen1}, Theorem \ref{thm:Chen2}]\label{mainthm:functor_category_version}
Let $\A$ be an additive category with weak-kernels and $\X$ its contravariantly finite full subcategory.
Suppose that $\pd_{\X}(\A(-,M)|_{\X})<\infty$ for any $M\in\A$
and  $\pd_{\A}(F)<\infty$ for any $F\in\mod(\A/[\X])$.
Then the canonical inclusion $\X\hookrightarrow\A$ induces a triangle equivalence $\bar{Q}:\ds(\A)\to\ds(\X)$.
\end{mainthm}

Our second result is an application of Theorem \ref{mainthm:functor_category_version}, which provides examples of singularly equivalent categories.
We denote by $\widehat{\X}$ the full subcategory of $\C$ consisting of objects $M$ which admit an exact sequence
$$0\to X_n\to X_{n-1}\to \cdots\to X_0\to M\to 0$$
with $X_n,\dots, X_0\in\X$ for some $n\in\mathbb{Z}_{\geq 0}$.
Our result will be stated under the following condition
which is a generalization of the setting appearing in Auslander-Buchweitz theory (see Condition \ref{AB_condition} for details).
A map $f :N\to M$ in $\C$ is called an \textit{$\X$-epimorphism} if the induced map $\C(-,N)|_{\X}\xto{f\circ -}\C(-,M)|_{\X}$ is surjective.

\begin{condition}\label{condition}
Let $\C$ be an abelian category with enough projectives and let $\A\supseteq\X\supseteq\boldsymbol{\omega}$ be a sequence of full subcategories in $\C$
such that $\X$ and $\boldsymbol{\omega}$ are contravariantly finite in $\A$.
We consider the following conditions:
\begin{enumerate}
\item[(AB1)] If a morphism $f:N\to M$ in $\A$ is an $\boldsymbol{\omega}$-epimorphism, then the kernel of $f$ belongs to $\A$.
\item[(AB2)] $\Ext^i_\C(X,I)=0$ for any $X\in\X, I\in\boldsymbol{\omega}$ and $i>0$.
\item[(AB3)] For any $M\in\A$, there exists an exact sequence $0\to Y_M\to X_M\xto{f} M$ in $\A$ such that $f$ is a right $\X$-approximation of $M$ and $Y_M\in\widehat{\boldsymbol{\omega}}$.
\end{enumerate}
\end{condition}

For example, the classical Auslander-Buchweitz theory (Condition \ref{AB_condition}) provides us with the triple $(\C=\A, \X, \boldsymbol{\omega})$ satisfies the Condition \ref{condition}.
Note that, in contrary to Condition \ref{AB_condition}, they are not required that: $\boldsymbol{\omega}$ is a \textit{cogenerator of $\X$}; each morphism $f$ appearing in $0\to Y_M\to X_M\xto{f} M$ of (AB3) is surjective.

Since $\X/[\boldsymbol{\omega}]$ can be regarded as an analog of the costable category, we denote by
$$
\overline{\A} := \A/[\boldsymbol{\omega}]\ \ \textnormal{and}\ \ \overline{\X} := \X/[\boldsymbol{\omega}].
$$

Our main result is the following:

\begin{mainthm}[Theorem \ref{thm:singular_equivalence_from_AB_approximation}]
\label{mainthm:singular_equivalence_from_AB_approximation}
Under Condition \ref{condition},
the canonical inclusion $\overline{\X}\hookrightarrow\overline{\A}$
induces a triangle equivalence $\ds(\overline{\A})\xto{\sim}\ds(\overline{\X})$.
\end{mainthm}

Typical examples satisfying Condition \ref{condition} come from cotilting theory.
Let us recall the notion of cotilting subcategories of $\C$.
For a subcategory $\X$ of $\C$, we denote by $^\perp \X$ the full subcategory of $\C$ of objects $M$ with $\Ext^i_\C(M,X)=0$ for any $i>0$ and $X\in\X$.

\begin{definition}
Let $\C$ be an abelian category with enough projectives.
A full subcategory $\T$ of $\C$ is called a \textit{cotilting subcategory of $\C$}, if it satisfies the following conditions:
\begin{itemize}
\item There exists an integer $n\in\mathbb{Z}_{\geq 0}$ such that $\id I\leq n$ for any $I\in\T$;
\item $\Ext^i_\C(I,J)=0$ for any $I, J\in\T$ and $i>0$;
\item For each $M\in{^\perp \T}$, there exists an exact sequence
$$0\to M\to I\to M'\to 0$$
with $I\in\T$ and $M'\in{^\perp \T}$.
\end{itemize}
We call an object $T\in\C$ a \textit{cotilting object} if $\add T$ is a cotilting subcategory of $\C$.
\end{definition}

The following result is immediate from Theorem \ref{mainthm:singular_equivalence_from_AB_approximation}.

\setcounter{maincor}{2}
\begin{maincor}[Corollary \ref{cor:singular_equivalence_from_cotilting}]
\label{maincor:singular_equivalence_from_cotilting}
Let $\A$ be an abelian category with enough projectives and $\T$ its contravariantly finite cotilting subcategory.
Then the canonical inclusion $\overline{^\perp \T}\hookrightarrow \overline{\A}$
induces a triangle equivalence $\ds(\overline{\A})\xto{\sim}\ds(\overline{^\perp \T})$.
\end{maincor}

As examples of Corollary \ref{maincor:singular_equivalence_from_cotilting}, we have the followings: 

\begin{example}\label{ex:CP}
\begin{itemize}
\item[(a)] Let $\Lambda $ be a finite dimensional $k$-algebra over a field $k$ and $T$ a cotilting $\Lambda $-module.
Then the canonical inclusion $\overline{^\perp T}\hookrightarrow\overline{\mod}\Lambda $
induces a triangle equivalence $\ds(\overline{\mod}\Lambda )\xto{\sim}\ds(\overline{^\perp T})$.
\item[(b)] Let $R$ be a commutative Cohen-Macaulay ring with a canonical $R$-module $\omega$
and $\mathsf{CM}R$ the full subcategory of maximal Cohen-Macaulay $R$-modules. 
Then the canonical inclusion $\overline{\mathsf{CM}}R\hookrightarrow\overline{\mod} R$
induces a triangle equivalence $\ds(\overline{\mod} R)\xto{\sim}\ds(\overline{\mathsf{CM}}R)$.
\end{itemize}
\end{example}

Theorem \ref{mainthm:singular_equivalence_from_AB_approximation} also provides an alternative proof for Matsui-Takahashi's theorem \cite[Thm. 5.4(3)]{MT} (Corollary \ref{cor:MT}):
For an Iwanaga-Gorenstein ring $\Lambda $, the canonical inclusion $\underline{\mathsf{CM}}\Lambda \hookrightarrow \underline{\mod}\Lambda $ induces a triangle equivalence $\ds(\underline{\mod}\Lambda )\xto{\sim}\ds(\underline{\mathsf{CM}}\Lambda )$.

\subsection*{Notation and convention}
Throughout the paper all categories and functors are assumed to be additive. 
The set of morphisms $M\rightarrow N$ in a category $\A$ is denoted by $\A(M,N)$.
Morphisms are composed from right-to-left.
Let $\X$ be a subcategory of $\A$.
We denote by $\A/[\X]$ the ideal quotient category of $\A$ modulo
the ideal $[\X]$ of $\A$ consisting of all morphisms which factor through an object in $\X$.
For each $M\in\A$, we denote by $\add M$ the full subcategory consisting of direct summands of a finite direct sum of $M$
and we abbreviate $\A/[M]$ to indicate $\A/[\add M]$.

The word ring and algebra always mean ring with a unit and finite dimensional algebra over a field $k$, respectively.
Let $A$ be a ring.
The symbol $\mod A$ denotes the category of finitely presented right $A$-modules.
We denote by $\Hom_A(M,N)$ the morphism-set from $M$ to $N$ instead of $(\mod A)(M,N)$.
The full subcategory of projective (resp. injective) modules in $\mod A$ will be denoted by $\proj A$ (resp. $\inj A$).
The projective (resp. injective) dimension of right $A$-module $M$ will be denoted by $\pd_A(M)$ (resp. $\id_A(M)$).

\section{A functor category version of Chen's theorem}\label{sect:2}
The aim of this section to provide a sufficient condition for an additive category $\A$ and its subcategory $\X$ so that the canonical inclusion $\X\hookrightarrow\A$ induces a triangle equivalence $\ds(\A)\simeq \ds(\X)$, which generalizes Xiao-Wu Chen's theorem.

The category $\mod\A$ is not necessarily abelian,
however, if every morphism in $\A$ has weak-kernels, then $\mod\A$ is abelian (\cite[Thm. 1.4]{Fre}).
Since we are interested in the case that $\mod\A$ is abelian,
\textit{throughout this section, let $\A$ be an additive category with weak-kernels and $\X$ its contravariantly finite full subcategory.}
Then, the canonical functor $Q:\mod\A\to\mod\X$ induces an equivalence
\begin{equation}\label{Serre3}
\frac{\mod\A}{\mod(\A/[\X])}\xto{\sim}\mod\X.
\end{equation}
Moreover, by \cite[Thm. 3.2]{Miy}, it induces a triangle equivalence
$$\frac{\Db(\mod\A)}{\mathsf{D}^{\textnormal{b}}_{\A/{[\X]}}(\mod\A)}\xto{\sim}\Db(\mod\X).$$
Then we have the following commutative diagram
$$
\xymatrix{
\mod(\A/[\X])\ar@{^{(}->}[d]\ar@{^{(}->}[r]&\mod\A\ar@{^{(}->}[d]\ar[r]^{Q}&\mod\X\ar[d]\\
\mathsf{D}^{\textnormal{b}}_{\A/{[\X]}}(\mod\A)\ar@{^{(}->}[r]&\Db(\mod\A)\ar[r]^{Q'}&\Db(\mod\X)
}
$$
where the arrows of the shape $\hookrightarrow$ denote canonical inclusions, and $Q'$ is the functor induced from $Q$.
Note that $\mathsf{D}^{\textnormal{b}}_{\A/[\X]}(\mod\A)$ is the thick subcategory of $\mathsf{D}^{\textnormal{b}}(\mod\A)$ containing $\mod(\A/[\X])$.
The following lemma gives a natural sufficient condition so that the canonical functor $\Db(\mod\A)\to\Db(\mod\X)$ induces a triangle functor $\ds(\A)\to\ds(\X)$.

\begin{lemma}\label{lem:Chen1}
The following conditions are equivalent:
\begin{itemize}
\item[\textnormal{(i)}] $\pd_\X(\A(-,M)|_\X)<\infty$ for any $M\in\A$;
\item[\textnormal{(ii)}]The canonical functor $Q':\Db(\mod\A)\to\Db(\mod\X)$ restricts to $Q':\Kb(\proj\A)\to\Kb(\proj\X)$.
\end{itemize}
If this is the case, we have an induced triangle functor $\bar{Q}:\ds(\A)\to\ds(\X)$.
\end{lemma}
\begin{proof}
(i) $\Leftrightarrow$ (ii): Since the functor $Q':\Db(\mod\A)\to\Db(\mod\X)$ restricts to $Q'|_{\mod\A}=Q:\mod\A\to\mod\X$, the condition (i) holds if and only if $Q'(\proj\A)\subseteq \Kb(\proj\X)$ if and only if the condition (ii) holds.

The latter statement follows from the universality of the Verdier quotient.
\end{proof}

Since our aim is to compare the singularity categories $\ds(\A)$ and $\ds(\X)$,
it is natural to assume that the equivalent conditions in Lemma \ref{lem:Chen1} are satisfied.
Our main result gives a necessary and sufficient condition so that the canonical inclusion $\X\hookrightarrow\A$ induces a triangle equivalence $\ds(\A)\xto{\sim}\ds(\X)$.

\begin{theorem}\label{thm:Chen2}
We assume that $\pd_\X(\A(-,M)|_\X) <\infty$ for any $M\in\A$.
Then the following conditions are equivalent:
\begin{itemize}
\item[\textnormal{(i)}] $\pd_{\A}(F)<\infty$ for any $F\in\mod(\A/[\X])$;
\item[\textnormal{(ii)}] The induced functor $\bar{Q}:\ds(\A)\to\ds(\X)$ is a triangle equivalence.
\end{itemize}
\end{theorem}

To prove Theorem \ref{thm:Chen2}, we firstly show Proposition \ref{prop:Verdier} in a more general framework:
Let $\T$ be a triangulated category with a translation $[1]$.
For a class $\S$ of objects in $\T$, we denote by $\tri \S$ the smallest triangulated full subcategory of $\T$ containing $\S$.
For two classes $\U$ and $\V$ of objects in
$\T$, we denote by $\U *\V$ the class of objects X occurring in a triangle $U\to X\to V\to U[1]$
with $U\in\U$ and $V\in\V$. Note that the operation $*$ is associative by the octahedral axiom.

\begin{proposition}\label{prop:Verdier}
Let $\U$ and $\V$ be triangulated full subcategories of $\T$ and consider the Verdier quotients with respect to them:
\begin{equation*}
\U\to\T\xto{Q_1}\T/\U\textnormal{\ \ \ and\ \ \ }\V\to\T\xto{Q_2}\T/\V.
\end{equation*}
Then, there exist natural triangle equivalences
$$\frac{\T/\U}{\tri(Q_1\V)}\simeq\frac{\T}{\tri(\U, \V)}\simeq \frac{\T/\V}{\tri(Q_2\U)},$$
where $Q_1\V$ is the full subcategory of $\T/\U$ consisting of objects isomorphic to $Q_1V$ for some $V\in\V$,
and the symbol $Q_2\U$ is used in a similar meaning.
\end{proposition}
\begin{proof}
We shall show an equality $\tri(Q_1\V)=\tri(\U,\V)/\U$, where $\tri(\U,\V)$ denotes the smallest triangulated full subcategory of $\T$ containing $\U$ and $\V$.
We set $\S :=\U\cup \V$. Obviously we have $Q_1\S=Q_1\V$.
Since $\tri (\U,\V)=\bigcup_{n\geq 0}\S^{*n}$, we have the following equalities:
\begin{equation*}
\tri(\U,\V)/\U = Q_1\Bigg(\bigcup_{n\geq 0}\S^{*n}\Bigg)
= \bigcup_{n\geq 0}(Q_1\S)^{*n}= \bigcup_{n\geq 0}(Q_1\V)^{*n} = \tri(Q_1\V).
\end{equation*}
Hence we have a desired triangle equivalence
$
\displaystyle\frac{\T/\U}{\tri(Q_1\V)}=\frac{\T/\U}{\tri(\U,\V)/\U}\xto{\sim}\frac{\T}{\tri(\U,\V)}.
$
\end{proof}

Now we are ready to prove Theorem \ref{thm:Chen2}.

\begin{proof}[Proof of Theorem \ref{thm:Chen2}]
Apply Proposition \ref{prop:Verdier} for $\T= \Db (\mod \A), \U=\Db_{\mod(\A/[\X])} (\mod \A)$ and $\V= \Kb(\proj \A)$. 
Then $\T/\U = \Db (\mod \X)$ and $\T/\V= \ds (\A)$. 
The assumption gives $Q_1\V= \Kb (\proj \X)$. 
Hence $\frac{\T/\U}{Q_1\V} = \ds (\X)$. 
Thus we have a triangle equivalence $\ds (\X) \simeq  \frac{\ds(\A)}{\tri (Q_2\U)}$. 
This shows the condition (i) is equivalent to $Q_2\U=0$, namely $\U \subset \V$, which is nothing but the condition (ii).
\end{proof}

We end this section with recovering the following Chen's theorem as a special case of Theorem \ref{thm:Chen2} and Lemma \ref{lem:Chen1}.

\begin{example}\cite[Thm. 1.3]{Che} \textnormal{(see also \cite[Thm. 5.2]{PSS}, \cite[Prop. 3.3]{KY})}\label{thm:singular_equivalence_for_idem}
Let $\Lambda $ be a Noetherian ring and $e$ its idempotent.
Assume that $\pd_{e\Lambda e}(\Lambda e)<\infty$.
Then the canonical inclusion $e\Lambda e\hookrightarrow \Lambda $ induces a triangle functor $\bar{Q}:\ds(\Lambda )\to\ds(e\Lambda e)$,
and the following are equivalent:
\begin{itemize}
\item [(i)] $\pd_\Lambda (M)<\infty$ for any $M\in\mod(\Lambda /\Lambda e\Lambda )$;
\item [(ii)] The induced functor $\bar{Q}: \ds(\Lambda )\xto{\sim}\ds(e\Lambda e)$ is a triangle equivalence.
\end{itemize}
\end{example}

\section{Sufficient conditions for singular equivalence}\label{sect:3}
The aim of this section is to construct a singular equivalence from our generalized Auslander-Buchweitz condition (Condition \ref{condition}).
First we introduce some terminology.
Let $\C$ denote an abelian category
and let $\C\supseteq \A\supseteq \B$ be a sequence of full subcategories of $\C$.
We call the kernel of a $\B$-epimorphism the \textit{$\B$-epikernel}, for short.
\textit{We assume that $\A$ is closed under $\B$-epikernels and $\B$ is contravariantly finite in $\A$.}
Then the ideal-quotient category $\A/[\B]$ admits weak-kernels.
In fact,
for a morphism $\alpha :M\to L$ of $\A$,
we obtain its weak-kernel as follows:
We take a right $\B$-approximation $\beta :B_L\to L$ of $L$,
and consider an induced exact sequence
$$0\to N\xto{(\gamma\ \delta)}M\oplus B_L\xto{\tiny{\begin{pmatrix}
\alpha  \\
\beta
\end{pmatrix}}} L$$
in $\C$.
Since $\A$ is closed under $\B$-epikernels and the morphism $\tiny{\begin{pmatrix}
\alpha  \\
\beta
\end{pmatrix}}$ is an $\B$-epimorphism,  we have $N\in\A$.
It is basic that the morphism $\gamma$ is a weak-kernel of $\alpha$ in $\A/[\B]$.

\subsection{Singular equivalences from Auslander-Buchweitz approximation}\label{subsect:proof}
In this subsection, we give a proof of  the following main theorem.

\begin{theorem}
\label{thm:singular_equivalence_from_AB_approximation}
Under Condition \ref{condition},
the canonical inclusion $\overline{\X}\hookrightarrow\overline{\A}$
induces a triangle equivalence $\ds(\overline{\A})\xto{\sim}\ds(\overline{\X})$.
\end{theorem}

Let $\C$ be an abelian category with enough projectives and consider a sequence
$\A\supseteq\X\supseteq\boldsymbol{\omega}$ of full subcategories in $\C$
such that $\X$ and $\boldsymbol{\omega}$ are contravariantly finite in $\A$.
We always assume (AB1) in Condition \ref{condition}.

\begin{proposition}\label{prop:Serre}
The ideal-quotient $\overline{\A}$ admits weak-kernels and $\overline{\X}$ is its contravariantly finite full subcategory.
Moreover, the canonical inclusion $\overline{\X}\hookrightarrow \overline{\A}$ induces the following equivalence
$$\frac{\mod\overline{\A}}{\mod(\A/[\X])}\xto{\sim}\mod\overline{\X}.$$
\end{proposition}
\begin{proof}
Since $\A$ is closed under $\boldsymbol{\omega}$-epikernels, $\overline{\A}$ admits weak-kernels.
Since $\X$ is contravariantly finite in $\A$, so is $\overline{\X}$ in $\overline{\A}$.
Note that there exists an equivalence $\A/[\X]\simeq \overline{\A}/[\overline{\X}]$.
By (\ref{Serre3}), we have a desired equivalence.
\end{proof}

To prove that the inclusion $\overline{\X}\hookrightarrow\overline{\A}$
induces a triangle functor $\ds(\overline{\A})\to\ds(\overline{\X})$,
we shall check a sufficient condition given in Lemma \ref{lem:Chen1}.

\begin{lemma}\label{lem:adjoint}
Assume (AB2) and (AB3). Let $X\in\X$ be given.
Then,
\begin{itemize}
\item[\textnormal{(a)}] One has $\Ext^i_\C(X,I)=0$ for any $I\in\widehat{\boldsymbol{\omega}}$ and $i>0$.
\item[\textnormal{(b)}] Every morphism $f:X\to I$ with $I\in\widehat{\boldsymbol{\omega}}$ factors through an object in $\boldsymbol{\omega}$.
\end{itemize}
\end{lemma}
\begin{proof}
We only show the assertion (b).
Since $I\in\widehat{\boldsymbol{\omega}}$, there exists an exact sequence $0\to I'\to W\to I\to 0$ with $W\in\boldsymbol{\omega}$ and $I'\in\widehat{\boldsymbol{\omega}}$.
Applying $\C(X,-)$, by (a), we conclude that $f$ factors through $W$.
\end{proof}

\begin{proposition}\label{prop:adjoint}
Assume (AB2) and (AB3).
Then the canonical inclusion $\textnormal{inc}:\overline{\X}\hookrightarrow \overline{\A}$ admits a right adjoint $R$.
Moreover, we have $\pd_{\overline{\X}}(\overline{\A}(-,M)|_{\overline{\X}})=0$ for any $M\in\A$.
\end{proposition}
\begin{proof}
The proof is similar to one given in \cite[Ch. V Prop. 1.2]{BR}, but our situation is slightly different from that in loc. cit. So we include a detailed  proof.
By (AB3), for each $M\in\A$,
there exists an exact sequence in $\A$
$$0\to Y_M\to X_M\xto{\alpha} M$$
with $\alpha$ a right $\X$-approximation of $M$ and $Y_M\in\widehat{\boldsymbol{\omega}}$.
We shall show that the morphism $\overline{\X}(X,X_M)\xto{\alpha\circ -}\overline{\A}(X,M)$ is a functorial isomorphism in $X\in\X$.
Its surjectivity is clear, since $\alpha$ is a right $\X$-approximation.
To show its injectivity, take a morphism $h\in\X(X,X_M)$ such that $\alpha\circ h$ factors through an object $I$ of $\boldsymbol{\omega}$.
Thus we have the following commutative diagram:
$$
\xymatrix@R=18pt{
&&X\ar[d]_h\ar[r]^{h'}&I\ar[d]^{h''}\\
0\ar[r]&Y_M\ar[r]&X_M\ar[r]^\alpha&M
}
$$
Since $\alpha$ is a right $\X$-approximation, there exists a morphism $\alpha':I\to X_M$ such that $\alpha\alpha'=h''$.
The morphism $h-\alpha' h'$ factors through $Y_M\in\widehat{\boldsymbol{\omega}}$.
By Lemma \ref{lem:adjoint}(ii), this implies that $h-\alpha' h'$ factors through $\boldsymbol{\omega}$.
Hence $h$ factors through $\boldsymbol{\omega}$.
By the Yoneda lemma, the assignment $M\mapsto X_M$ gives rise to a functor $R:\A\to\X$.
The bifunctorial isomorphism $\overline{\X}(X,R(M))\xto{\alpha\circ -}\overline{\A}(X,M)$ says the pair of functors $(\textnormal{inc},R)$ forms an adjoint pair.

The latter statement is obvious.
\end{proof}

\begin{proposition}\label{prop:projective dimension}
Let $\B$ be a contravariantly finite full subcategory of $\A$
and assume that $\A$ is closed under $\B$-epikernels.
Let $F\in\mod(\A/[\B])$ be given.
Then there exists an exact sequence
\begin{equation}\label{A-projective_resolution}
0\to N\xto{g} M\xto{f} L
\end{equation}
in $\A$ which satisfies the following conditions:
\begin{itemize}
\item[\textnormal{(a)}] The morphism $f$ is a $\B$-epimorphism;
\item[\textnormal{(b)}] The induces sequence
\begin{equation*}
0\to\A(-,N)\xto{g\circ -}\A(-,M)\xto{f\circ -}\A(-,L)\to F\to 0
\end{equation*}
is exact.
\end{itemize}
In particular, $\pd_\A(F)\leq 2$.
\end{proposition}
\begin{proof}
First $F$ is a finitely presented $\A$-module.
Indeed,
a right $\B$-approximation $B_Y\to Y$ of any $Y\in\A$ induces a projective presentation
$$\A(-,B_Y)\to\A(-,Y)\to\A/[\B](-,Y)\to 0$$
of the $\A$-module $\A/[\B](-,Y)$.
This shows that $\A/[\B](-,Y)$ belongs to $\mod\A$, hence so does $F$.

Thus we have a projective presentation
$\A(-,M)\xto{f\circ -}\A(-,L)\to F\to 0$
of the $\A$-module $F$.
Since $F$ vanishes on $\B$, the induced morphism $f$ is a $\B$-epimorphism.
Thus we have an exact sequence $0\to N\xto{g} M\xto{f}L$ in $\A$.
Applying the Yoneda embedding, we have a projective resolution $0\to\A(-,N)\to\A(-,M)\to\A(-,L)\to F\to 0$ of the $\A$-module $F$.
\end{proof}

Let $M\in\A$ and $f:B_M\to M$ be a right $\B$-approximation of $M$.
Then we write $\Omega_\B(M):=\Ker f$.
We define $\Omega_\B^n(M)$ inductively for $n\geq 1$.
We prove the following key-proposition which generalizes the well-known result given in \cite[Prop. 4.1, 4.2]{AR74} and \cite[Prop. 1.2]{AR96}.
The proof is similar but a bit different from the original ones.

\begin{proposition}\label{prop:key}
For $F\in\mod(\A/[\B])$,
the exact sequence (\ref{A-projective_resolution}) in Proposition \ref{prop:projective dimension} induces a projective resolution
\begin{equation}\label{A/X-projective_resolution}
\xymatrix@R=4pt@M=6pt{
\cdots\ar[r]&\A/[\B](-,\Omega^2_\B(N))\ar[r]^{\Omega^2_\B g\circ -}\ar[r]&\A/[\B](-,\Omega^2_\B(M))\ar[r]^{\Omega^2_\B f\circ -}&\A/[\B](-,\Omega^2_\B(L))&&\\
\ar[r]&\A/[\B](-,\Omega_\B(N))\ar[r]^{\Omega_\B g\circ -}\ar[r]&\A/[\B](-,\Omega_\B(M))\ar[r]^{\Omega_\B f\circ -}&\A/[\B](-,\Omega_\B(L))&&\\
\ar[r]&\A/[\B](-,N)\ar[r]^{g\circ -}&\A/[\B](-,M)\ar[r]^{f\circ -}&\A/[\B](-,L)\ar[r]&F\ar[r]&0
}
\end{equation}
of the $\A/[\B]$-module $F$.
\end{proposition}
\begin{proof}
For the sequence (\ref{A-projective_resolution}),
we take right $\B$-approximations $\alpha_L:B_L\to L$ and $\alpha_N:B_N\to N$.
Since the morphism $f$ is $\B$-epimorphism, we have a morphism $\beta: B_L\to M$ such that $\alpha_L=f\circ \beta$.
The induced morphism $\alpha_M:=\tiny{\begin{pmatrix}
\beta  \\
g\alpha_N
\end{pmatrix}}:B_M:=B_L\oplus B_N\to M$ is a right $\B$-approximation of $M$.
Since $\A$ is closed under $\B$-epikernels, we have the following commutative diagram in $\A$:
\begin{equation*}
\xymatrix@R=14pt{
&0\ar[d]&0\ar[d]&0\ar[d]&\\
0\ar[r]&\Omega_\B(N)\ar[r]^{\Omega_\B g}\ar[d]&\Omega_\B(M)\ar[r]^{\Omega_\B f}\ar[d]&\Omega_\B(L)\ar[d]&\\
0\ar[r]&B_N\ar[d]^{\alpha_N}\ar[r]&B_M\ar[d]^{\alpha_M}\ar[r]&B_L\ar[d]^{\alpha_L}\ar[r]&0\\
0\ar[r]&N\ar[r]^g&M\ar[r]^f&L&
}
\end{equation*}
where all columns and rows are exact, and the middle row splits.
Applying the Yoneda embedding and the Snake Lemma, we have the following commutative diagram in $\mod\A$.
\begin{equation*}
\xymatrix@R=14pt{
&0\ar[d]&0\ar[d]&0\ar[d]&&\\
0\ar[r]&\A(-,\Omega_\B(N))\ar[d]\ar[r]&\A(-,\Omega_\B(M))\ar[d]\ar[r]&\A(-,\Omega_\B(L))\ar[d]&&\\
0\ar[r]&\A(-,B_N)\ar[d]^{\alpha_N\circ -}\ar[r]&\A(-,B_M)\ar[d]^{\alpha_M\circ -}\ar[r]&\A(-,B_L)\ar[d]^{\alpha_L\circ -}\ar[r]&0&\\
0\ar[r]&\A(-,N)\ar[d]\ar[r]^{g\circ -}&\A(-,M)\ar[d]\ar[r]^{f\circ -}&\A(-,L)\ar[d]\ar[r]&F\ar@{=}[d]\ar[r]&0\\
&\A/[\B](-,N)\ar[d]\ar[r]&\A/[\B](-,M)\ar[d]\ar[r]&\A/[\B](-,L)\ar[d]\ar[r]&F\ar[r]&0\\
&0&0&0&&
}
\end{equation*}
In particular, we have an exact sequence
$$
\xymatrix@C=20pt@R=3pt{
0\ar[r]&\A(-,\Omega_\B(N))\ar[r]&\A(-,\Omega_\B(M))\ar[r]&\A(-,\Omega_\B(L))&&\\
\ar[r]^{\delta\ \ \ \ \ \ \ }&\A/[\B](-,N)\ar[r]&\A/[\B](-,M)\ar[r]&\A/[\B](-,L)\ar[r]&F\ar[r]&0
}$$
in $\mod\A$.
We have an exact sequence $0\to\Omega_\B(N)\to\Omega_\B(M)\xto{\Omega_\B f}\Omega_\B(L)$ such that $\Omega_\B f$ is an $\B$-epimorphism.
Inductively, we have a desired projective resolution of the $\A/[\B]$-module $F$.
\end{proof}

\begin{lemma}\label{lem:singular complete}
Under Condition \ref{condition},
\begin{itemize}
\item[\textnormal{(a)}] For any $L\in\A$, there exists $n\geq 0$ such that $\Omega^n_\X(L)\in\X$.
\item[\textnormal{(b)}] For each $F\in\mod(\A/[\X])$, we have $\pd_{\A/[\X]}(F)<\infty$.
\end{itemize}
\end{lemma}
\begin{proof}
(a) For an object $L\in\A$, due to (AB3), we get an exact sequence
$$0\to Y\to X_0\xto{f_0}L$$
such that $f_0$ is a right $\X$-approximation of $L$ and $Y\in\widehat{\boldsymbol{\omega}}$.
Since $Y\in\widehat{\boldsymbol{\omega}}$, we get an exact sequence
$$0\to I_{n}\to I_{n-1}\xto{f_{n-1}}\cdots \to I_1\xto{f_1}X_0\xto{f_0} L$$
with $I_i\in\boldsymbol{\omega}$ for $1\leq i\leq n$.
By Lemma \ref{lem:adjoint}, each morphism
$f_i:I_i\to\Im f_i$ is a right $\X$-approximation of $\Im f_i$ for each $1\leq i\leq n$ hence $I_{n}=\Omega_\X^n(L)\in\X$.

(b) We consider the projective resolution (\ref{A/X-projective_resolution}) of the $\A/[\X]$-module $F$ given in Proposition \ref{prop:key} by setting $\B:=\X$.
Then the assertion follows from (a), since $\A/[\X](-,\Omega_\X^n(L))=0$.
\end{proof}

\begin{proposition}\label{prop:singular complete}
Under Condition \ref{condition},
for each $F\in\mod(\A/[\X])$, one has
$\pd_{\overline{\A}}(F)<\infty$.
\end{proposition}
\begin{proof}
Since $\pd_{\A/[\X]}(F)<\infty$ by Lemma \ref{lem:singular complete} and the canonical inclusion $\iota:\mod(\A/[\X])\hookrightarrow\mod(\overline{\A})$ is exact, it is enough to check the case of $F=\A/[\X](-,M)$ for some $M\in\A$.
By (AB3), there exists an exact sequence $0\to Y_M\xto{g} X_M\xto{f} M$ in $\A$ with $f$ a right $\X$-approximation of $M$ and $Y_M\in\widehat{\boldsymbol{\omega}}$.
Applying the Yoneda embedding yields a projective resolution
$$0\to\A(-,Y_M)\xto{g\circ -}\A(-,X_M)\xto{f\circ -}\A(-,M)\to\A/[\X](-,M)\to 0$$
of the $\A$-module $\A/[\X](-,M)$.
Applying Proposition \ref{prop:key} to $\B:=\boldsymbol{\omega}$,
we have a projective resolution of the $\overline{\A}$-module $\A/[\X](-,M)$:
$$
\xymatrix@C=26pt@R=3pt{
\cdots\ar[r]&\overline{\A}(-,\Omega_{\boldsymbol{\omega}}(Y_M))\ar[r]^{\Omega_{\boldsymbol{\omega}} g\circ -}&\overline{\A}(-,\Omega_{\boldsymbol{\omega}}(X_M))\ar[r]^{\Omega_{\boldsymbol{\omega}} f\circ -}&\overline{\A}(-,\Omega_{\boldsymbol{\omega}}(M))&&\\
\ar[r]&\overline{\A}(-,Y_M)\ar[r]^{g\circ -}&\overline{\A}(-,X_M)\ar[r]^{f\circ -}&\overline{\A}(-,M)\ar[r]&A/[\X](-,M)\ar[r]&0.
}$$
Since $Y_M\in\widehat{\boldsymbol{\omega}}$, one has $\Omega_{\boldsymbol{\omega}}^n(Y_M)\in\boldsymbol{\omega}$ for some $n\geq 0$.
Thus $\overline{\A}(-,\Omega^n_\X(Y_M))=0$
and hence $\pd_{\overline{\A}}(\A/[\X](-,M))<\infty$.
\end{proof}

We are ready to prove Theorem \ref{thm:singular_equivalence_from_AB_approximation}.

\begin{proof}[Proof of Theorem \ref{thm:singular_equivalence_from_AB_approximation}]
By Lemma \ref{lem:Chen1} and Proposition \ref{prop:adjoint},
the canonical inclusion $\overline{\X}\hookrightarrow\overline{\A}$ induces a triangle functor $\bar{Q}:\ds(\overline{\A})\to\ds(\overline{\X})$.
By Theorem \ref{thm:Chen2} and Proposition \ref{prop:singular complete}, the triangle functor $\bar{Q}$ is an equivalence.
\end{proof}

\subsection{Singular equivalences from cotilting objects}
In this subsection we construct a singular equivalence from a given cotilting subcategory, using Theorem \ref{thm:singular_equivalence_from_AB_approximation}.
We denote by $\mathsf{P}(\C)$ (resp. $\mathsf{GP}(\C)$) the full subcategory of $\C$
consisting of projective (resp. Gorenstein projective) objects.
We abbreviate $\Omega M:=\Omega_{\mathsf{P}(\C)}M$ for each $M\in\C$
and denote by $\Omega^n \A$ the full subcategory of $\C$ consisting of objects isomorphic to $\Omega^n M$ for some $M\in\A$.
Moreover we define $\Omega^- M$ to be the kernel of a left $\mathsf{P}(\C)$-approximation of $M$.
Inductively we define $\Omega^{-n}M$ for any $n\geq 1$.

\begin{corollary}
\label{cor:singular_equivalence_from_cotilting}
Let $\A$ be an abelian category with enough projectives and $\T$ its contravariantly finite cotilting subcategory.
Then the canonical inclusion $\overline{^\perp \T}\hookrightarrow \overline{\A}$
induces a triangle equivalence $\ds(\overline{\A})\xto{\sim}\ds(\overline{^\perp \T})$.
\end{corollary}
\begin{proof}
Setting $\X:={^\perp\T}$ and $\boldsymbol{\omega}:=\T$, we shall show that the sequence $\A\supseteq \X\supseteq \boldsymbol{\omega}$ satisfies conditions (AB1)-(AB3).
The condition (AB1) is obvious, because $\A=\C$.
The condition (AB2) holds by definition.

(AB3): By \cite[Thm. 1.1]{ABu}, for any $M\in\widehat{\X}$, there exists an exact sequence
$$0\to Y_M\to X_M\to M\to 0$$
with $Y_M\in\widehat{\omega}$ and $X_M\in\X$.
It remains to show $\widehat{\X}=\A$.
Since there exists an integer $n\geq 0$ such that $\id I\leq n$ for all $I\in\boldsymbol{\omega}$, it follows that $\Omega^n M\in\X$ holds for all $M\in\A$.
This shows $\widehat{\X}=\A$.
Thanks to Theorem \ref{thm:singular_equivalence_from_AB_approximation}, we have a desired triangle equivalence.
\end{proof}

\subsection{Matsui-Takahashi's Singular equivalence}
We provide an alternative proof for Matsui-Takahashi's singular equivalence.

\begin{definition}\label{def:resolving}
Let $\C$ be an abelian category with enough projectives.
A full subcategory $\A$ of $\C$ is called \textit{quasi-resolving} if it is closed under kernels of epimorphisms and contains all projectives. A quasi-resolving subcategory is called \textit{resolving} if it is closed under extensions and direct summands.
\end{definition}

\begin{corollary}\textnormal{\cite[Thm. 5.4(3)]{MT}}
\label{cor:MT}
Let $\A$ be a quasi-resolving subcategory of an abelian category $\C$ with enough projectives.
Assume that $\A$ together with an integer $n\in\mathbb{Z}_{\geq 0}$ satisfies the condition
\begin{equation}
\Omega^n\A \textit{\ is contained in\ }\mathsf{GP}(\C) \textit{\ and closed under cosyzygies} \tag{$*$}
\end{equation}and set $\X:=\Omega^n\A$.
Then the canonical inclusion $\underline{\X}\hookrightarrow\underline{\A}$ induces a triangle equivalence $\ds(\underline{\A})\xto{\sim} \ds(\underline{\X})$.
\end{corollary}
\begin{proof}
Setting $\X:=\Omega^n\A$ and $\boldsymbol{\omega}:=\mathsf{P}(\C)$, we shall show that the sequence $\A\supseteq\X\supseteq\boldsymbol{\omega}$ of subcategories in $\C$ satisfies the conditions (AB1)-(AB3).
(AB1): Since $\mathsf{P}(\C)$-epikernels are epimorphisms, the condition (AB1) follows from the definition of quasi-resolving subcategories.

(AB2): Since $\X\subseteq \mathsf{GP}(\C)$, we have $\X\subseteq {^\perp\boldsymbol{\omega}}$.

(AB3): Let $M\in\A$.
By the condition $(*)$, we have an exact sequence
$$0\to G\to P_{n-1}\to \cdots P_0\to M\to 0$$
with $G\in\X$ and $P_{n-1},\cdots, P_0\in\mathsf{P}(\C)$.
Since $G\in\mathsf{GP}(\C)$, we have an exact sequence
$$0\to G\xto{g_n} Q_{n-1}\xto{g_{n-1}} \cdots\to Q_0\xto{g_0} \Omega^{-n}(G)\to 0$$
with the canonical morphisms $\Im g_i\to Q_i$ being left $\mathsf{P}(\C)$-approximations for each $1\leq i\leq n$.
Thus we have the following chain map, where $\Omega^{-n}(G)\in\Omega^n\A =\X$ by the condition $(*)$.
\begin{equation*}
\xymatrix{
0\ar[r]&G\ar@{=}[d]\ar[r]& Q_{n-1}\ar[d]\ar[r]& \cdots\ar[r] &Q_0\ar[d]\ar[r]& \Omega^{-n}(G)\ar[r]\ar[d]& 0\\
0\ar[r]&G\ar[r]& P_{n-1}\ar[r]& \cdots\ar[r] &P_0\ar[r]& M\ar[r]& 0
}
\end{equation*}
By taking the mapping cone of the above chain map, we have an exact sequence
$$0\to G\to Q_{n-1}\oplus G\to Q_{n-2}\oplus P_{n-1}\to \cdots\to Q_0\oplus P_1\to\Omega^{-n}(G)\oplus P_0\xto{} M\to 0.$$
Since the left-most morphism $G\to Q_{n-1}\oplus G$ is a split-monomorphism,
we have the following exact sequence
\begin{equation}\label{seq:1}
0\to Q_{n-1}\to Q_{n-2}\oplus P_{n-1}\to \cdots\to Q_0\oplus P_1\to\Omega^{-n}(G)\oplus P_0\xto{f} M\to 0.
\end{equation}
Obviously $\Ker f\in\widehat{\boldsymbol{\omega}}$ holds.
The exact sequence $0\to \Ker f\to \Omega^{-n}(G)\oplus P_0\xto{f} M\to 0$ is a desired one. Indeed, $f$ is a right $\X$-approximation by Lemma \ref{lem:adjoint}.
\end{proof}

Recall that an additive category $\A$ with weak-kernels is said to be \textit{Iwanaga-Gorenstein} if $\id_\A (\A(-,M)), \id_{\A^{\op}}(\A(M,-))<\infty$ for any $M\in\A$.
Typical examples of Iwanaga-Gorenstein rings are
finite dimensional selfinjective algebras over a field $k$ and
commutative Gorenstein rings of finite Krull dimension.
As an obvious consequence of Corollary \ref{cor:singular_equivalence_from_cotilting} or \ref{cor:MT}, we have:

\begin{example}\label{ex:IG}
Let $\Lambda $ be an Iwanaga-Gorenstein ring with $\id_\Lambda (\Lambda )=n$ and $\mathsf{CM}\Lambda :={^\perp \Lambda }$.
Then the canonical inclusion $\underline{\mathsf{CM}}\Lambda \hookrightarrow\underline{\mod}\Lambda $
induces a triangle equivalence $\ds(\underline{\mod}\Lambda )\xto{\sim}\ds(\underline{\mathsf{CM}}\Lambda )$.
\end{example}

\section{More Results and Examples}
In this section, we provide further investigations on Condition \ref{condition}.
First we give sufficient conditions so that $\X/[\omega]$ is Iwanaga-Gorenstein and of finite global dimension, respectively.

\begin{theorem}\label{thm:global_dimension}
Let $\Lambda $ be a finite dimensional algebra and $T\in\mod \Lambda $ a cotilting module.
We set $\underline{^\perp T}:={^\perp T}/[\Lambda]$ and $\overline{^\perp T}:={^\perp T}/[T]$.
Then the followings hold:
\begin{itemize}
\item[\textnormal{(a)}] If $\Lambda $ is Iwanaga-Gorenstein, then so is $\overline{^\perp T}$.
Moreover, one has $\id_{(\overline{^\perp T})} F\leq 3\max\{\pd_\Lambda  T, \id_\Lambda \Lambda \}$ for any projective $(\overline{^\perp T})$-module $F$.
\item[\textnormal{(b)}] If $\gd \Lambda =n$, then we have $\gd (\overline{^\perp T})\leq 3n-1$.
\end{itemize}
\end{theorem}

The assertion (b) can be found in \cite[Thm. 6.1]{Kim}.
Let us recall from \cite[Thm. 3.4]{INP} (see also \cite{Eno, Jia}), there exist Auslander-Reiten translations on $^\perp T$, that is, mutually equivalences
$$\tau: \underline{^\perp T}\xto{\sim}\overline{^\perp T}\textnormal{\ \ and\ \ }\tau^-:\overline{^\perp T}\xto{\sim}\underline{^\perp T}.$$
Moreover, they induce functorial isomorphisms
$$D\Ext^1_A(M,N)\cong \underline{^\perp T}(\tau^- N,M)\cong \overline{^\perp T}(N,\tau M)$$
in $M,N\in{^\perp T}$ which are known as Auslander-Reiten dualities, where $D:=\Hom_k(-.k)$.

\begin{proof}[Proof of Theorem \ref{thm:global_dimension}]
(a) Since there exists an equivalence $\overline{^\perp T}\xto{\sim}\underline{^\perp T}$, we shall show that $\underline{^\perp T}$ is Iwanaga-Gorenstein.
Thanks to Auslander-Reiten duality, every injective $(\underline{^\perp T})$-module is of the form $\Ext^1_\Lambda (-,M)$ for some $M\in{^\perp T}$.
Since $T$ is a cotilting module, we get an exact sequence $0\to M\to T'\to N\to 0$ with $T'\in\add T$ and $N\in{^\perp T}$.
The induced sequence
$$0\to \Hom_\Lambda (-,M)\to\Hom_\Lambda (-,T')\to\Hom_\Lambda (-,N)\to \Ext^1_\Lambda (-,M)\to 0$$
gives a projective resolution of  $(^\perp T)$-module $\Ext^1_\Lambda (-,M)$.
By Proposition \ref{prop:key}, we have a projective resolution
\begin{equation}\label{pr_for_IG}
\xymatrix@C=20pt@R=3pt{
\cdots\ar[r]&\underline{^\perp T}(-,\Omega_\Lambda (M))\ar[r]&\underline{^\perp T}(-,\Omega_\Lambda (T'))\ar[r]&\underline{^\perp T}(-,\Omega_\Lambda (N))&&\\
\ar[r]^{\delta\ \ \ \ \ \ \ }&\underline{^\perp T}(-,M)\ar[r]&\underline{^\perp T}(-,T')\ar[r]&\underline{^\perp T}(-,N)\ar[r]&\Ext^1_\Lambda (-,M)\ar[r]&0
}
\end{equation}
of the $(\underline{^\perp T})$-module $\Ext^1_\Lambda (-,M)$.
Since $\Lambda $ is Iwanaga-Gorenstein, $T$ is a tilting module, in particular $\pd_\Lambda (T)<\infty$.
Thus there exists an integer $n\geq 0$ such that $\Omega^n_\Lambda (T')\in\proj \Lambda $.
Hence every injective $(\underline{^\perp T})$-module $\Ext^1_\Lambda (-,M)$ is of finite projective dimension.
Next we shall show that every projective $(\underline{^\perp T})$-module $\underline{^\perp T}(-,M)$ is of finite injective dimension.
Considering the first syzygy of $M$, namely an exact sequence $0\to \Omega_\Lambda M\to P\to M\to 0$ with $P\in\proj \Lambda $, we get an injective resolution
\begin{equation}\label{pr_for_IG2}
0\to \underline{^\perp T}(-,M)\to \Ext^1_\Lambda (-,\Omega_\Lambda M)\to\Ext^1_\Lambda (-,P)\to\Ext^1_\Lambda (-,M)\to \cdots
\end{equation}
of the $(\underline{^\perp T})$-module $\underline{^\perp T}(-,M)$.
Since $\Lambda $ is Iwanaga-Gorenstein, we have $\id_\Lambda P<\infty$.
We have thus concluded that $\overline{^\perp T}$ is Iwanaga-Gorenstein.
The latter formula follows from the sequence (\ref{pr_for_IG}) and (\ref{pr_for_IG2}).

(b) We shall show that $\gd(\underline{^\perp T})\leq 3n-1$.
Let $F\in \mod(\underline{^\perp T})$ with a projective presentation $\underline{^\perp T}(-,M)\to\underline{^\perp T}(-,L)\to F\to 0$.
Since $F$ vanishes on $\proj \Lambda $, the corresponding morphism $f:M\to L$ is an epimorphism in $\mod \Lambda $.
Since $^\perp T$ is closed under epimorphisms, we have an exact sequence
$0\to N\to M\to L\to 0$ in $^\perp T$ which induces a projective resolution
$$
\xymatrix@C=20pt@R=3pt{
\cdots\ar[r]&\underline{^\perp T}(-,\Omega_\Lambda (N))\ar[r]&\underline{^\perp T}(-,\Omega_\Lambda (M))\ar[r]&\underline{^\perp T}(-,\Omega_\Lambda (L))&&\\
\ar[r]&\underline{^\perp T}(-,N)\ar[r]&\underline{^\perp T}(-,M)\ar[r]&\underline{^\perp T}(-,L)\ar[r]&F\ar[r]&0
}$$
of the $(\underline{^\perp T})$-module $F$.
The assumption $\gd \Lambda =n$ implies $\Omega^n_\Lambda  (L)\in\proj \Lambda $.
Hence $\pd_{(\underline{^\perp T})}F\leq 3n-1$.
\end{proof}

Theorem \ref{thm:global_dimension} contains the following well-known result.

\begin{example}\cite[Prop. 10.2]{AR74}
Let $\Lambda $ be a finite dimensional algebra with $\gd \Lambda =n$.
Then we have $\gd (\underline{\mod} \Lambda )\leq 3n-1.$
\end{example}

Next we explain that (AB1)-(AB3) in Condition \ref{condition} are satisfied in the classical Auslander-Buchweitz theory:
Let $\C$ be an abelian category with enough projectives and $\X\supseteq \boldsymbol{\omega}$ a sequence of full subcategories in $\C$.
We say that $\boldsymbol{\omega}$ is a \textit{cogenerator of $\X$} if, for each $X\in\X$, there exists an exact sequence $0\to X\to I\to X'\to 0$ with $I\in\boldsymbol{\omega}, X'\in\X$.

\begin{condition}\cite[p. 9, 17]{ABu}
\label{AB_condition}
For a sequence $\X\supseteq\boldsymbol{\omega}$ of full subcategories in $\C$,
we consider the following conditions:
\begin{itemize}
\item $\widehat{\X}=\C$;
\item $\X$ is closed under direct summands and extension;
\item $\Ext^i_\C(X,I)=0$ for any $X\in\X, I\in\boldsymbol{\omega}$ and $i>0$;
\item $\boldsymbol{\omega}$ is a cogenerator of $\X$ which is closed under direct summands.
\end{itemize}
\end{condition}

Under these conditions, it is known that, for each $M\in\C$, there exists an exact sequence
\begin{equation}\label{seq:AB-approximation}
0\to Y_M\to X_M\xto{} M\to0
\end{equation}
with $X_M\in\X, Y_M\in\widehat{\boldsymbol{\omega}}$ \cite[Thm. 1.1]{ABu}.
The sequence ($\ref{seq:AB-approximation}$) is called the \textit{Auslander-Buchweitz approximation of $M$}.
As a benefit of our \textit{generalized Auslander-Buchweitz approximation} in (AB3), we shall show Proposition \ref{prop:torsion}.
Notice that, in the proposition, the subcategory $\omega$ is not necessarily a cogenerator of $\X$, and right $\X$-approximations of objects of $\A$ appearing in (AB3) are not necessarily surjective.

\begin{proposition}\label{prop:torsion}
Let $\A$ be an abelian category with enough projectives and $\X\supseteq \boldsymbol{\omega}$ a sequence of full subcategories of $\A$.
Suppose that $\X$ is a torsion class of $\A$ and $\boldsymbol{\omega}$ is contravariantly finite in $\A$ and satisfies $\Ext_\A^i(X,I)=0$ for any $X\in\X, I\in\boldsymbol{\omega}$ and $i> 0$. Then the sequence $\A\supseteq \X\supseteq \boldsymbol{\omega}$ satisfies (AB1)-(AB3).
\end{proposition}
\begin{proof}
The conditions (AB1) and (AB2) are obvious.
Since $\X$ is a torsion class, for any $M\in\A$ there exists an exact sequence $0\to X\to M$ with $X\in\X$, hence (AB3) holds.
\end{proof}

We end this section by giving examples of singularly equivalent categories using Corollary \ref{cor:singular_equivalence_from_cotilting}.

\begin{example}
Fix an integer $n\in\mathbb{Z}_{>0}$.
Let $\Lambda $ be the algebra defined by the following quiver with relations.
$$
\xymatrix{
1\ar@/^12pt/[r]^\alpha&2,\ar@/^12pt/[l]^\beta &\langle (\alpha\beta)^n\alpha\rangle
}
$$
We describe the Auslander-Reiten quiver of $\Lambda $.
Since $\Lambda $ is a Nakayama algebra, an indecomposable module is determined by the pair $(m,l)$ of the socle $l$ and the Loewy length $l$.
We shall denote the module by $[m]_l$.
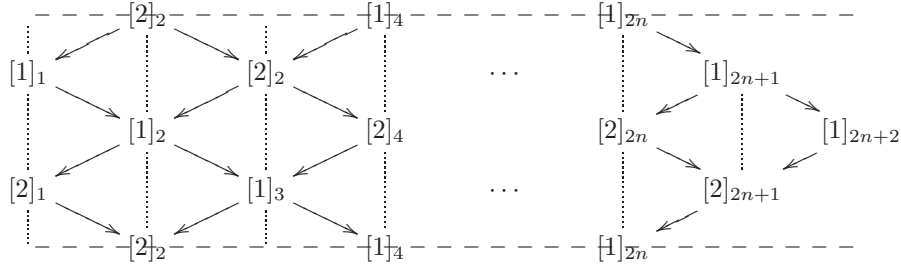
\begin{equation*}
{\small
\xymatrix@!C=21pt@R=6pt{
\ar@{--}[rrrrrrr]\ar@{.}[d]&[2]_2\ar@{.}[dd]\ar[rd]\ar[ld]&\ar@{.}[d]&[1]_4\ar@{.}[dd]\ar[ld]&&[1]_{2n}\ar@{.}[dd]\ar[rd]&&\\
[1]_1\ar@{.}[dd]\ar[rd]&&[2]_2\ar@{.}[dd]\ar[rd]\ar[ld]&&\cdots&&[1]_{2n+1}\ar@{.}[dd]\ar[rd]\ar[ld]&\\
&[1]_2\ar@{.}[dd]\ar[rd]\ar[ld]&&[2]_4\ar@{.}[dd]\ar[ld]&&[2]_{2n}\ar@{.}[dd]\ar[rd]&&[1]_{2n+2}\ar[ld]\\
[2]_1\ar@{.}[d]\ar[rd]&&[1]_3\ar@{.}[d]\ar[rd]\ar[ld]&&\cdots&&[2]_{2n+1}\ar[ld]&\\
\ar@{--}[rrrrrrr]&[2]_2&&[1]_4&&[1]_{2n}&&
}}
\end{equation*}
We can easily check that the module $T:=[1]_1\oplus [1]_{2n+2}$ is a cotilting module of $\id_\Lambda (T)= 1$.
Due to Corollary \ref{cor:singular_equivalence_from_cotilting}, we conclude that $\overline{\mod}\Lambda :=(\mod \Lambda )/[T]$ is singularly equivalent to $\overline{^\perp T}:=({^\perp T})/[T]$.
Their Auslander-Reiten quivers are described as follows:
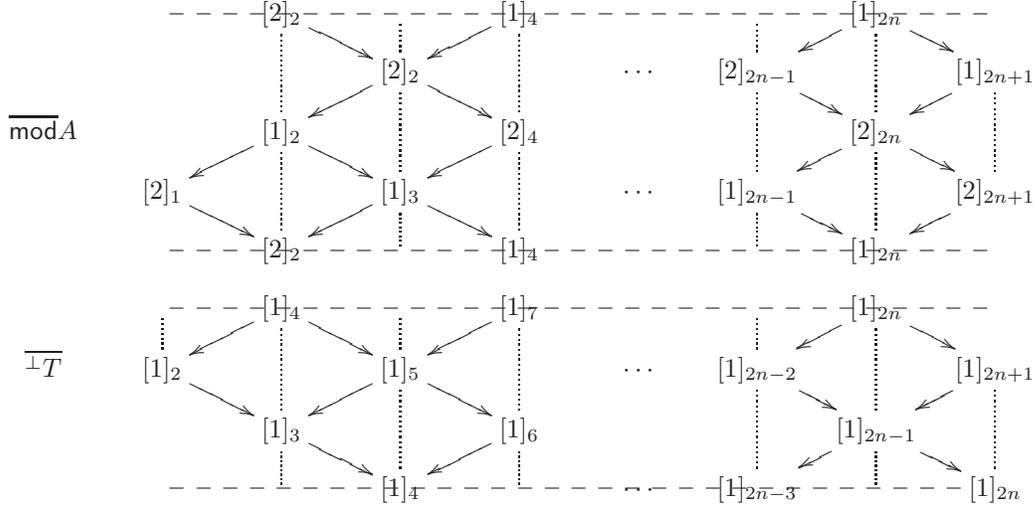
\begin{equation*}
{\small
\xymatrix@!C=21pt@R=6pt{
&\ar@{--}[rrrrrrr]&[2]_2\ar@{.}[dd]\ar[rd]&\ar@{.}[d]&[1]_4\ar@{.}[dd]\ar[ld]&&\ar@{.}[d]&[1]_{2n}\ar@{.}[dd]\ar[ld]\ar[rd]&\\
&&&[2]_2\ar@{.}[dd]\ar[rd]\ar[ld]&&\cdots&[2]_{2n-1}\ar@{.}[dd]\ar[rd]&&[1]_{2n+1}\ar@{.}[dd]\ar[ld]\\
\overline{\mod} A&&[1]_2\ar@{.}[dd]\ar[rd]\ar[ld]&&[2]_4\ar@{.}[dd]\ar[ld]&&&[2]_{2n}\ar@{.}[dd]\ar[ld]\ar[rd]&\\
&[2]_1\ar[rd]&&[1]_3\ar@{.}[d]\ar[rd]\ar[ld]&&\cdots&[1]_{2n-1}\ar@{.}[d]\ar[rd]&&[2]_{2n+1}\ar[ld]\\
&\ar@{--}[rrrrrrr]&[2]_2&&[1]_4&&&[1]_{2n}&&\\
&\ar@{--}[rrrrrrr]\ar@{.}[d]&[1]_4\ar@{.}[dd]\ar[rd]\ar[ld]&\ar@{.}[d]&[1]_7\ar@{.}[dd]\ar[ld]&&\ar@{.}[d]&[1]_{2n}\ar@{.}[dd]\ar[ld]\ar[rd]&\\
\overline{^\perp T}&[1]_2\ar[rd]&&[1]_5\ar@{.}[dd]\ar[rd]\ar[ld]&&\cdots&[1]_{2n-2}\ar@{.}[dd]\ar[rd]&&[1]_{2n+1}\ar@{.}[dd]\ar[ld]\\
&&[1]_3\ar@{.}[d]\ar[rd]&&[1]_6\ar@{.}[d]\ar[ld]&&&[1]_{2n-1}\ar@{.}[d]\ar[rd]\ar[ld]&\\
&\ar@{--}[rrrrrrr]&&[1]_4&&\cdots&[1]_{2n-3}&&[1]_{2n}
}}
\end{equation*}
where the dotted lines stand for natural mesh relations.

\begin{claim*}
If $n=1$, both $\overline{\mod}\Lambda $ and $\overline{^\perp T}$ are of finite global dimension, otherwise they are non Iwanaga-Gorenstein.
\end{claim*}
\begin{proof}
We only check the case of $n\geq 2$.
By calculations,
the injective $(\overline{^\perp T})$-module $D\overline{^\perp T}([1]_3,-)$ has the following projective resolution:
$$\cdots\to P_5\to P_3\to P_{2n+1}\to P_{2n-1}\to P_{2n+1}\to P_3\to P_4\to P_{2n+1}\to I_3\to 0,$$
where we set $I_3:=D\overline{^\perp T}([1]_3,-)$ and
$P_l:=\overline{^\perp T}(-,[1]_l)$ for each $1\leq l\leq 2n+1$.
We notice that $\Omega^2I_3\cong \Omega^8I_3$.
Hence $\overline{^\perp T}$ is non Iwanaga-Gorenstein.
It remains to check the assertion for $\overline{\mod}\Lambda $.
We denote by $Q: \mod(\overline{\mod}\Lambda )\to \mod(\overline{^\perp T})$ the canonical functor.
There exists an injective object $J\in\inj(\overline{\mod}\Lambda )$ such that $QJ\cong I_3$.
If $\overline{\mod}\Lambda $ is Iwanaga-Gorenstein, then $J$ is of finite projective dimension.
Moreover, since $Q$ is exact and preserves projectives, it turns out that $I_3$ is of finite projective dimension.
This is a contradiction.
\end{proof}
\end{example}



\begin{thebibliography}{99}
\bibitem[AKL]{AKL}
L.~Angeleri.~H.~H\"ugel, S.~Koenig, Q.~Liu, \emph{Jordan-H\"older theorems for derived module categories of piecewise hereditary algebra}.
 J. Algebra 352 (2012), 361--381.
 \bibitem[ABr]{ABr}
M.~Auslander, M.~Bridger, \emph{Stable module theory}.
Memoirs of the American
Mathematical Society, No. 94. American Mathematical Society, Providence, R.I., 1969.
\bibitem[ABu]{ABu}
M.~Auslander, R-O.~Buchweitz, \emph{The homological theory of maximal Cohen-Macaulay approximations}.
Colloque en l'honneur de Pierre Samuel (Orsay, 1987). M\'em. Soc. Math. France (N.S.) No. 38 (1989), 5--37.
\bibitem[APT]{APT92}
M.~Auslander, M.~I.~Platzeck, G.~Todorov, \emph{Homological theory of idempotent ideals}.
Trans.
Amer. Math. Soc. 332 (1992), no. 2, 667–692.
\bibitem[AR74]{AR74}
M.~Auslander, I.~Reiten, \emph{Stable equivalence of dualizing R-varieties}.
Adv. in Math. 12 (1974) No. 3, 306-366.
\bibitem[AR96]{AR96}
M.~Auslander, I.~Reiten, \emph{$D\operatorname{Tr}$-periodic modules and functors}.
in: Bautista, Raymundo
(ed.) et al., ``Representation theory of algebras. Seventh international conference, August
22-26, 1994, Cocoyoc, Mexico.'' CMS Conf. Proc. 18, 39-50, American Mathematical
Society, Providence 1996.
\bibitem[BR]{BR}
A.~Beligiannis, I.~Reiten, \emph{Homological and homotopical aspects of torsion theories}.
Mem. Amer. Math. Soc. 188 (2007), no. 883, viii+207 pp. 
\bibitem[Buc86]{Buc86}
R-O.~Buchweitz, \emph{Maximal Cohen-Macaulay modules and Tate-cohomology
over Gorenstein rings}.
Unpublished manuscript, 1986.
\bibitem[Buc97]{Buc97}
R-O. Buchweitz, \emph{Finite representation type and periodic Hochschild (co-)homology}.
Trends in the representation theory of finite-dimensional algebras (Seattle, WA, 1997), 81--109, Contemp. Math., 229, Amer. Math. Soc., Providence, RI, 1998.
\bibitem[Che]{Che}
X.~Chen, \emph{Unifying two results of Orlov on singularity categories}.
Abh. Math. Semin. Univ. Hambg. 80 (2010), no. 2, 207--212.
\bibitem[CPS]{CPS}
E.~Cline, B.~Parshall, and L.~Scott, \emph{Finite-dimensional algebras and highest weight categories}.
J. Reine Angew. Math. 391 (1988), 85-99.
\bibitem[DR]{DR}
V.~Dlab, C.~M.~Ringel, \emph{The module theoretical approach to quasi-hereditary algebras}.
Representations of algebras
and related topics (Kyoto, 1990), London Math. Soc. Lecture Note Ser., vol. 168,
Cambridge Univ. Press, Cambridge, 1992, pp. 200–224.
\bibitem[Eno]{Eno}
H.~Enomoto, \emph{Classifications of exact structures and Cohen-Macaulay-finite algebras}. Adv. Math. 335 (2018), 838-877.
\bibitem[Fre]{Fre}
P.~Freyd, \emph{Representations in abelian categories}.
1966 Proc. Conf. Categorical Algebra (La Jolla, Calif., 1965) pp. 95--120 Springer, New York.
\bibitem[Iya18]{Iya18}
O.~Iyama, \emph{Tilting Cohen-Macaulay representations}.
arXiv:1805.05318v1, May 2018.
\bibitem[INP]{INP}
O.~Iyama, H.~Nakaoka, Y.~Palu, \emph{Auslander--Reiten theory in extriangulated categories}.
arXiv:1805.03776.
\bibitem[Jia]{Jia}
P.~Jiao, \emph{The generalized Auslander-Reiten duality on an exact category}.
arXiv:1609.07732.
\bibitem[KY]{KY}
M.~Kalck, D.~Yang, \emph{Relative singularity categories I: Auslander resolutions}.
Adv. Math. 301 (2016), 973--1021.
\bibitem[KV]{KV}
B.~Keller, D.~Vossieck, \emph{Sous les cat\'egories d\'eriv\'ees}.
C. R. Acad. Sci. Paris 305 (1987), 225-228.
\bibitem[Kim]{Kim}
Y.~Kimura, \emph{Tilting theory of preprojective algebras and $c$-sortable elements}.
 J. Algebra 503 (2018), 186--221. 
\bibitem[MT]{MT}
H.~Matsui, R.~Takahashi, \emph{Singularity categories and singular equivalences for resolving subcategories}.
Math. Z. 285 (2017), no. 1-2, 251--286.
\bibitem[Miy]{Miy}
J.~Miyachi, \emph{Localization of triangulated categories and derived categories}.
J. Algebra 141 (1991), no. 2, 463--483.
\bibitem[Orl04]{Orl04}
D.~O.~Orlov, \emph{Triangulated categories of singularities and D-branes in Landau-Ginzburg models}.
Proc. Steklov Inst. Math. 246 (2004), no. 3, 227–248.
\bibitem[Orl09]{Orl09}
D.~O.~Orlov, \emph{Derived categories of coherent sheaves and triangulated categories of singularities}.
Algebra, arithmetic, and geometry: in honor of Yu. I. Manin. Vol. II, 503--531, Progr. Math., 270, Birkhäuser Boston, Inc., Boston, MA, 2009.
\bibitem[PSS]{PSS}
C.~Psaroudakis, \O.~Skarts{\ae}terhagen, \O.~Solberg, \emph{Gorenstein categories, singular equivalences and finite generation of cohomology rings in recollements}.
Trans. Amer. Math. Soc. Ser. B 1 (2014), 45--95.
\bibitem[Ric]{Ric}
J.~Rickard, \emph{Derived categories and stable equivalence}.
J. Pure Appl. Algebra 61 (1989), no. 3, 303--317.
\bibitem[Zim]{Zim}
A.~Zimmermann, \emph{Representation theory. A homological algebra point of view}.
Algebra and Applications, 19. Springer, Cham, 2014.
\bibitem[ZZ]{ZZ}
G.~Zhou, A.~Zimmermann, \emph{On singular equivalences of Morita type}.
J. Algebra 385 (2013), 64--79.
\end{thebibliography}
\end{document}